\documentclass[twoside,11pt,reqno]{amsart}
\usepackage{amsmath,amssymb,amscd,mathrsfs,epic,wasysym,latexsym,tikz,mathrsfs,cite,hyperref}
\usepackage{stmaryrd}
\usepackage{pb-diagram}
\usepackage[matrix,arrow]{xy}

\makeatletter

\numberwithin{equation}{section}

\newtheorem{Proposition}[equation]{Proposition}
\newtheorem{Lemma}[equation]{Lemma}
\newtheorem{Theorem}[equation]{Theorem}
\newtheorem{Corollary}[equation]{Corollary}
\theoremstyle{definition}  
\newtheorem{Definition}[equation]{Definition}
\newtheorem{Remark}[equation]{Remark}

\let\<\langle
\let\>\rangle

\newcommand\Comment[2][\relax]{\space\par\medskip\noindent%
   \fbox{\begin{minipage}{\textwidth}\textbf{Comment\ifx\relax#1\else---#1\fi}\newline%
        #2\end{minipage}}\medskip
}

\def\b1{\text{\boldmath$1$}}

\def\fa{\mathfrak{a}}

\def\a{\mathfrak{a}}

\def\m{\mathfrak{m}}

\def\pmod#1{\text{ }(\text{\rm mod } #1)\,}

\newcommand{\End}{\operatorname{End}}

\newcommand{\Z}{\mathbb{Z}}

\newcommand{\0}{{\bar 0}}
\renewcommand{\1}{{\bar 1}}
\def\eps{{\varepsilon}}
\def\phi{{\varphi}}

\newcommand{\funF}{{\mathcal F}}

\newcommand{\la}{\lambda}

\newcommand{\Om}{\Omega}

\newcommand{\de}{\delta}
\newcommand{\De}{\Delta}

\newcommand{\rad}{{\mathrm {rad}\,}}

\renewcommand{\mod}{\bmod \,}

\newcommand{\EZig}{Z}

\def\col{{\operatorname{col}}}

\def\b{\mathfrak{b}}
\def\k{\Bbbk}

\def\spa{\operatorname{span}}

\def\op{{\mathrm{op}}}

\def\mod#1{#1\!\operatorname{-mod}}

\def\iso{\stackrel{\sim}{\longrightarrow}}

{\catcode`\|=\active
  \gdef\set#1{\mathinner{\lbrace\,{\mathcode`\|"8000%
  \let|\midvert #1}\,\rbrace}}
}
\def\midvert{\egroup\mid\bgroup}

\colorlet{darkgreen}{green!50!black}
\tikzset{dots/.style={very thick,loosely dotted},
         greendot/.style={fill,circle,color=darkgreen,inner sep=1.5pt,outer sep=0},
         blackdot/.style={fill,circle,color=black,inner sep=1.5pt,outer sep=0},
         graydot/.style={fill,circle,color=gray,inner sep=1.1pt,outer sep=0}
}
\def\greendot(#1,#2){\node[greendot] at(#1,#2){}}
\def\blackdot(#1,#2){\node[blackdot] at(#1,#2){}}
\def\graydot(#1,#2){\node[graydot] at(#1,#2){}}

\newenvironment{braid}{
  \begin{tikzpicture}[baseline=6mm,black,line width=1pt, scale=0.32,
                      draw/.append style={rounded corners},
                      every node/.append style={font=\fontsize{5}{5}\selectfont}]%
  }{\end{tikzpicture}
}

\def\Grid(#1,#2){
  \draw[very thin,gray,step=2mm] (0,0)grid(#1,#2);
  \draw[very thin,darkgreen,step=10mm] (0,0)grid(#1,#2);
}

\newcommand\Tableau[2][\relax]{
  \begin{tikzpicture}[scale=0.5,draw/.append style={thick,black}]
    \ifx\relax#1\relax%
    \else 
      \foreach\box in {#1} { \filldraw[blue!30]\box+(-.5,-.5)rectangle++(.5,.5); }
    \fi
    \newcount\row\newcount\col
    \row=0
    \foreach \Row in {#2} {
       \col=1
       \foreach\k in \Row {
          \draw(\the\col,\the\row)+(-.5,-.5)rectangle++(.5,.5);
          \draw(\the\col,\the\row)node{\k};
          \global\advance\col by 1
       }
       \global\advance\row by -1
    }
  \end{tikzpicture}
}

\newcommand\YoungDiagram[2][\relax]{
  \begin{tikzpicture}[scale=0.5,draw/.append style={thick,black}]
    \ifx\relax#1\relax%
    \else 
    \foreach\box in {#1} {
      \filldraw[blue!30]\box rectangle ++(1,1);
    }
    \fi
    \newcount\row
    \row=0
    \foreach \col in {#2} {
       \draw(1,\the\row)grid ++(\col,1);
       \global\advance\row by -1
    }
  \end{tikzpicture}
}

\begin{document}

\title[Based quasi-hereditary algebras]{{\bf Based quasi-hereditary algebras}}

\author{\sc Alexander Kleshchev}
\address{Department of Mathematics\\ University of Oregon\\
Eugene\\ OR 97403, USA}
\email{klesh@uoregon.edu}

\author{\sc Robert Muth}
\address{Department of Mathematics\\ Washington \& Jefferson College
\\
Washington\\ PA 15301, USA}
\email{rmuth@washjeff.edu}


\thanks{The first author was supported by the NSF grant No. DMS-1700905 and the DFG Mercator program through the University of Stuttgart. This work was also supported by the NSF under grant No. DMS-1440140 while both authors were in residence at the MSRI during the Spring 2018 semester. Additional support was provided by the National University of Singapore IMS, where both authors were in residence for the Representation Theory of Symmetric Groups and Related Algebras workshop, December 11--20, 2017.
}

\begin{abstract}
A notion of a split quasi-hereditary algebra has been defined by Cline, Parshall and Scott. Du and Rui  describe a based approach to split quasi-hereditary algebras. 
We develop this approach further to show that over a complete local Noetherian ring, one can achieve even stronger basis properties. This is important for {\em `schurifying'} quasi-hereditary algebras as developed in our subsequent work. The schurification procedure associates to an algebra $A$ a new algebra, which is the classical Schur algebra if $A$ is a field. Schurification produces interesting new quasi-hereditary and cellular algebras. It is important to work over an integral domain of characteristic zero, taking into account a super-structure on the input algebra $A$. So we pay attention to   super-structures on quasi-hereditary algebras and investigate a subtle {\em conforming} property of heredity data which is crucial to guarantee that the schurification of $A$ is quasi-hereditary if so is $A$. 
We establish a Morita equivalence result which allows us to pass to basic quasi-hereditary algebras {\em preserving conformity}. 

\end{abstract}

\maketitle

\section{Introduction}

Working over an arbitrary ground field, Cline, Parshall and Scott \cite{CPSCrelle} axiomatized the notion of a {\em highest weight category} and defined {\em quasi-hereditary algebras}. However, it is important to be able to work more generally over a reasonable commutative ring. This was pursued in \cite{CPS,DS,DR,Ro}. In particular, if $\k$ is a Noetherian ground ring, a notion of a split quasi-hereditary algebra has been defined in \cite{CPS}, cf. also \cite{Ro}. On the other hand, Du and Rui \cite{DR} described a based approach to split quasi-hereditary algebras, showing that it is equivalent to that of \cite{CPS} provided that $\k$ is Noetherian and local. 

The goal of this paper is to develop Du and Rui's approach further to show that over a complete local Noetherian ring, we can achieve even stronger basis properties, see Definition~\ref{DCC}. This is important for {\em `schurifying'} quasi-hereditary algebras as developed in \cite{greenThree}. 
The schurification procedure associates to a $\k$-algebra $A$ (with suitable subalgebra \(\fa\)) a new algebra $T^A_\fa(n,d)$, which is the classical Schur algebra if $A=\k$. Schurification often produces interesting new quasi-hereditary and cellular algebras which are important in representation theory of symmetric groups, Hecke algebras, classical Schur algebras, etc., see e.g. \cite{T,EK1,EK2}. 

It is clear from \cite{EK1,EK2} that to define many interesting quasi-hereditary algebras, it is important to work over an integral domain of characteristic zero, taking into account a super-structure on the input algebra $A$. Therefore we pay attention to   super-structures (as well as  $\Z$-gradings) on quasi-hereditary algebras. We investigate a subtle {\em conforming} property of heredity data, see Definition~\ref{D290517}. This is non-trivial only if the super-structure is non-trivial, and is crucial to guarantee that $T^A_\fa(n,d)$ is quasi-hereditary if $A$ is quasi-hereditary. 

We further establish some Morita equivalence results which sometimes allow us to pass to basic (or almost basic) quasi-hereditary algebras {\em preserving conformity}, see Theorem~\ref{MorBasic}. This is crucial for studying decomposition numbers and other properties of $T^A_\fa(n,d)$, see for example \cite{greenThree}. 

\section{Based quasi-hereditary algebras}\label{SQHA}

Throughout the paper $\k$ is always a commutative unital ring. Sometimes we will require more in which case this will be stated explicitly. 

\subsection{Algebras and modules}

Let $V$ be a {\em graded $\k$-supermodule}, i.e. $V$ is endowed with a $\k$-module decomposition  
$$V=\bigoplus_{n\in\Z,\,\eps\in\Z/2}V^n_\eps.
$$ 
We set $V^n:=V^n_\0\oplus V^n_\1$ and $V_\eps:=\bigoplus_{n\in\Z}V^n_\eps$. 
Then $V=\bigoplus_{n\in\Z} V^n$ is a grading, and $V=V_\0\oplus V_\1$ is a superstructure. For $v\in V_\eps$, we write $\bar v:=\eps$. 
Of course, the grading and/or the superstructure could be trivial, for example we could have $V=V^0_\0$.  

An element $v\in V$ is called homogeneous if $v\in V_\eps^m$ for some $\eps$ and $m$. We denote by $V_{\rm hom}$ the set of all non-zero homogeneous elements of $V$. For a subset $S\subseteq V_{\rm hom}$ and $\eps\in\Z/2$ 
we denote 
\begin{equation}\label{EH}
S_\eps:=S\cap  V_\eps.
\end{equation}

A map $f:V\to W$ of graded $\k$-supermodules is called {\em homogeneous} if $f(V^m_\eps)\subseteq W^m_\eps$ for all $m$ and $\eps$. 
Let
\begin{equation}\label{ER}
R:=\Z[q,q^{-1}][t]/(t^2-1), 
\end{equation} 
and denote the image of $t$ in the quotient ring by $\pi$, so that $\pi^\eps$ makes sense for $\eps\in\Z/2$. 
For $v\in V^n_\eps$, we write 
\begin{equation}\label{EDeg}
\deg(v):=q^n\pi^\eps.
\end{equation}

For a free $\k$-module $W$ of finite rank $d$, we write $d=\dim W$. A graded $\k$-supermodule $V$ is free of finite rank if each $V^n_\eps$ is free of finite rank and we have $V^n=0$ for almost all $n$. Let $V$ be a free graded $\k$-supermodule of finite rank. 
A {\em homogeneous basis} of $V$ is a $\k$-basis all of whose elements are homogeneous. The {\em graded dimension} of $V$ is 
$$
\dim^q_\pi V:=\sum_{n\in\Z,\, \eps\in\Z/2}(\dim V^n_\eps)q^n\pi^\eps\in R.
$$
   
A (not necessarily unital) $\k$-algebra $A$ is called a {\em graded $\k$-superalgebra}, if $A$ is a graded $\k$-supermodule and $A_{\eps}^nA_{\de}^m\subseteq A_{\eps+\de}^{n+m}$ for all $\eps,\de$ and $n,m$.  
By a {\em graded $A$-supermodule} we understand an  $A$-module $V$ which is a graded $\k$-supermodule and 
$A_{\eps}^nV_{\de}^m\subseteq V_{\eps+\de}^{n+m}$ for all $\eps,\de$ and $n,m$. 
We denote by $\mod{A}$ the category of all finitely generated graded $A$-supermodules and homogeneous $A$-homomorphisms. All ideals, subalgebras, submodules, etc. are assumed to be homogeneous. In particular the Jacobson ideal $J(A)$ is the intersection of the annihilators of all graded simple $A$-supermodules. 

Given a graded $A$-supermodule $V$, $n\in \Z$ and $\eps\in\Z/2\Z$, we denote by $q^n\pi^\eps V$ the graded $A$-supermodule which is the same as $V$ as an $A$-module but with $(q^n\pi^\eps V)^m_\de=V^{m-n}_{\de+\eps}$.


\subsection{Definition and first properties}\label{SSCC}
Let $A$ be a graded $\k$-superalgebra, and $I$ be a finite partially ordered set. A subset $\Omega\subseteq I$ is called an {\em upper set} if $i\in\Om$ and $j\geq i$ imply $j\in\Om$. Examples of upper sets are 
$$
I^{>i}:=\{j\in I\mid j>i\}\quad\text{and}\quad I^{\geq i}:=\{j\in I\mid j\geq i\}
$$
for a fixed $i\in I$.

\begin{Definition} \label{DCC} 
{\rm 
A {\em heredity data} on $A$ consist of a partially ordered set $I$ and finite sets $X=\bigsqcup_{i\in I}X(i)$ and $Y=\bigsqcup_{i\in I}Y(i)$ of non-zero homogeneous elements of $A$ with distinguished {\em initial elements} $e^i\in X(i)\cap  Y(i)$ for each $i\in I$.
For $i\in I$ and $\Om\subseteq I$, we set 
\begin{align*}
&Z(i):=X(i)\times Y(i),\ Z(\Om):=\textstyle \bigsqcup_{j\in \Om} Z(j),
\\ 
&Z^{>i}:=Z(I^{>i}), \  Z^{\geq i}:=Z(I^{\geq i}),  \\  
&A(\Om):=\spa\{xy \mid (x,y)\in \Om\},\\  
&A^{>i}:=A(Z^{>i}),\  A^{\geq i}:=A(Z^{\geq i}).  
\end{align*}
We require that the following axioms hold: 
\begin{enumerate}
\item[{\rm (a)}] $B:=\{x y \mid  (x,y)\in Z\}$ is a basis of $A$; 

\item[{\rm (b)}] For all $i\in I$, $x\in X(i)$, $y\in Y(i)$ and $a\in A$, we have
$$
a x \equiv \sum_{x'\in X(i)}l^x_{x'}(a)x' \pmod{A^{>i}}
\ \ \text{and}\ \ 
ya \equiv \sum_{y'\in Y(i)}r^y_{y'}(a)y' \pmod{A^{>i}}
$$
for some $l^x_{x'}(a),r^y_{y'}(a)\in\k$;

\item[{\rm (c)}] For all $i\in I$, we have  
\begin{align*}
&xe_i= x,\ e_ix= \de_{x,e_i}x,\ e_i y= y,\ ye_i= \de_{y,e_i}y &(x\in X(i),\ y\in Y(i));
\\
&e_jx=x\ \text{or}\ 0,\ ye_j=y\ \text{or}\ 0 &(x\in X,\ y\in Y,\ j\in I).
\end{align*}
\end{enumerate}
}
\end{Definition}

If $A$ is endowed with a heredity data $I,X,Y$, we call $A$ {\em based quasi-hereditary (with respect to the poset $I$)}, and refer to $B$ as a {\em heredity basis} of $A$.


\begin{Lemma} \label{LHI} 
If $\Om\subseteq I$ is an upper set, then $A(\Om)$ is the (two-sided) ideal generated by $\{e_i\mid i\in\Om\}$. 
\end{Lemma}
\begin{proof}
That $A(\Om)$ is an ideal is clear from Definition~\ref{DCC}(b). That $A(\Om)$ contains the ideal generated by $\{e_i\mid i\in\Om\}$ is now clear since $A(\Om)\supseteq \{e_i\mid i\in\Om\}$. The converse containment follows from $xy=xe_iy$ for $(x,y)\in Z(i)$, see Definition~\ref{DCC}(c). 
\end{proof}


\begin{Lemma} \label{L290417} 
Let $\Om,\Theta\subseteq I$ be upper sets.
\begin{enumerate}
\item[{\rm (i)}] $A(\Om) \subseteq A(\Theta)$ if and only if $\Om\subseteq \Theta$;
\item[{\rm (ii)}] $A(\Om) A(\Theta)\subseteq A(\Om)\cap A(\Theta)= A(\Om\cap\Theta)$.
\end{enumerate}
\end{Lemma}
\begin{proof}
(i) If $\Om \not\subseteq \Theta$ and $i\in \Om\setminus \Theta$, it follows from Definition~\ref{DCC}(a) that $xy\in A(\Om) \setminus A(\Theta)$ for all $(x,y)\in Z(i)$, i.e. $A(\Om) \not\subseteq A(\Theta)$. 
The converse is obvious. 

(ii) As $A(\Om)$, $A(\Theta)$ are ideals by Lemma~\ref{LHI}, the containment 
$A(\Om) A(\Theta)\subseteq A(\Om)\cap A(\Theta)$ is clear. The equality $A(\Om)\cap A(\Theta)= A(\Om\cap\Theta)$ comes from Definition~\ref{DCC}(a). 
\end{proof}

\begin{Lemma} \label{idemaction}
Let \(x \in X(i)\), \(y \in Y(i)\). If \(j \not \leq i\), then \(e_j x = y e_j = 0\).
\end{Lemma}
\begin{proof}
As \(e_j \in Y(j)\), we have by Definition~\ref{DCC}(b) that \(e_j x  \in A^{\geq j}\). Since \(x \notin A^{\geq j}\), we have that \(e_j x \neq x\), so Definition~\ref{DCC}(c) gives us \(e_jx = 0\). The proof of \(ye_j = 0\) is similar.
\end{proof}

\begin{Lemma} 
For any $i,j\in I$, we have $e_ie_j=\de_{i,j}e_i$. 
\end{Lemma}
\begin{proof}
Since $e_i\in I$, the equality $e_i^2=e_i$ comes from 
Definition~\ref{DCC}(c). Let $i\neq j$. By Definition~\ref{DCC}(c) again, we have that $e_ie_j$ is either $e_j$ or $0$ and on the other hand either $e_i$ or $0$. Since $e_i\neq e_j$ by Definition~\ref{DCC}(a), we deduce that $e_ie_j=0$. 
\end{proof}

Let $i\in I$, $x\in X(i)$ and $y\in Y(i)$. By Definition~\ref{DCC}(b), 
$$
\sum_{x'\in X(i)} l^x_{x'}(y)x' \equiv yx\equiv \sum_{y'\in Y(i)} r^y_{y'}(x)y'
\pmod{A^{>i}}.
$$
By Definition~\ref{DCC}(c), we have $x'= x'e_i$ and $y'= e_iy'$, so taking into account Definition~\ref{DCC}(a), we deduce that 
\begin{equation}\label{E210617}
yx\equiv f_i(y,x)e_i \pmod{A^{>i}}
\end{equation}
for some $f_i(y,x)\in\k$. 
This defines a function $f_i:Y(i)\times X(i)\to \k$. Note that  
\begin{equation}\label{E121217_3}
f_i(e_i,e_i)=1
\end{equation}
and 
\begin{equation}\label{E050817}
\text{$f_i(y,x)=0$\quad unless\quad  $\deg(x)\deg(y)=1$.}
\end{equation}

\begin{Definition} \label{DSB} 
{\rm \cite[1.2.1]{DR}} 
{\rm 
A graded $\k$-superalgebra $A$ is called {\em standardly based} with respect to a finite poset $I$ if it possesses a {\em standard basis}, i.e. a homogeneous basis of the form 
$$
\{b^i_{x,y}\mid i\in I,\ x\in X(i),\ y\in Y(i)\}
$$
for some index sets $X(i),Y(i)$ such that, setting $A^{>i}:=\spa\{b^j_{x,y}\mid j>i\}$, 
for all $a\in A$, $i\in I$, $x\in X(i)$, $y\in Y(i)$, we have
\begin{align*}
 a b^i_{x,y} &\equiv \sum_{x'\in X(i)}l^x_{x'}( a) b^i_{x',y}  \pmod{A^{>i}},
\\
 b^i_{x,y}  a &\equiv \sum_{y'\in Y(i)}r^y_{y'}( a) b^i_{x,y'} \pmod{ A^{>i}}
\end{align*}
for some $l^x_{x'}(a)\in \k$ independent of $y$ and $r^y_{y'}(a)\in\k$ independent of $x$. 
}
\end{Definition}

By \cite[(1.2.3)]{DR},
$$
b^i_{x,y}b^i_{x',y'}\equiv f_i(y,x')b^i_{x,y'}\pmod{A^{>i}}
$$ 
for some $f_i(y,x')\in \k$. 
The standardly based algebra is called standardly full-based if the $\k$-span of the elements $f_i(y,x)$, with $x\in X(i)$, $y\in Y(i)$, is $\k$. The following is clear using (\ref{E121217_3}):

\begin{Lemma} 
If $A$ is a based quasi-hereditary algebra then it is standardly full-based with $b^i_{x,y}=xy$ for all $i\in I$ and all $x\in X(i)$, $y\in Y(i)$.
\end{Lemma}

A homogeneous anti-involution $\tau$ on $A$ is called {\em standard} (with respect to $I,X,Y$) if for all $i\in I$ there is a bijection $X(i)\iso Y(i),\ x\mapsto y(x)$ such that $y(e_i)=e_i$ and  
\begin{equation}\label{E200517}
\tau(x)= y(x).
\end{equation}
For a standard anti-involution $\tau$, we have 
\begin{equation}\label{E230517_6}
\tau(xy(x'))= x'y(x)
\end{equation}
 and $\tau(e_i)= e_i$ for all $i\in I,\ x,x'\in X(i)$. 
If $\tau$ is a standard anti-involution on $A$ then $\{xy\mid (x,y)\in Z\}$ is a {\em cellular basis} of $A$ with respect to $\tau$, see \cite[(6.1.4)]{DR}.

\subsection{Standard modules}
\label{SSDelta}
Throughout the subsection, $A$ is a based quasi-hereditary $\k$-superalgebra with heredity data $I,X,Y$.  

Fix $i\in I$ and upper sets $\Om',\Om\subseteq I$ such that $\Om'\setminus \Om=\{i\}$. For example we could take $\Om'=I^{\geq i}$ and $\Om=I^{>i}$. 
Denote 
$$\tilde A:=A/A(\Om)\quad \text{and}\quad \tilde a:=a+A(\Om)\in \tilde A\qquad (a\in A).
$$ 
By inflation, $\tilde A$-modules will be automatically considered as $A$-modules. The {\em standard module $\De(i)$} and the {\em right standard module $\De^\op(i)$} are defined as
\begin{equation}\label{EDe}
\De(i):=\tilde A \tilde e_i\quad\text{and}\quad \De^\op(i):=\tilde e_i\tilde A.
\end{equation}

By Definition~\ref{DCC}, we have 
$$\De(i)=\spa\{\tilde x \mid x\in X(i)\}\quad \text{and}\quad \De^\op(i)=\spa\{\tilde y \mid y\in Y(i)\},$$ 
so $\De(i)$ and $\De^\op(i)$ can be defined respectively as free $\k$-modules with bases $\{v_x\mid x\in X(i)\}$ and $\{w_y\mid y\in Y(i)\}$ and the actions 
$$av_x=\sum_{x'\in X(i)}l_{x'}^x(a)v_{x'}\quad \text{and}\quad    
w_y a=\sum_{y'\in Y(i)}r_{y'}^y(a)w_{y'}
\qquad(a\in A).
$$ 
This implies in particular that the definition of $\De(i)$ and $\De^\op(i)$ does not depend on the choice of $\Om$ and $\Om'$ as long as $\Om'\setminus \Om=\{i\}$.

Note that $v_i:=v_{e_i}$ is a cyclic generator of $\De(i)$ such that 
\begin{equation}\label{E121217}
e_iv_i=v_i\quad \text{and}\quad xv_i=v_x\qquad(x\in X(i)).
\end{equation}
Moreover, 
\begin{equation}\label{E121217_1}
e_i v_x=0\qquad(x\in X(i)\setminus\{e_i\}).
\end{equation}
Taking into account Lemma \ref{idemaction}, we deduce that  $e_j\De(i)\neq 0$ implies $j\leq i$.  
Similar statements hold for $\De^\op(i)$. 
We have 
\begin{equation}\label{LEndDe}
\End_A(\De(i))\cong \End_{\tilde A}(\De(i))\cong \End_{\tilde A}(\tilde Ae_i,\De(i))\cong e_i\De(i)\cong 
\k.
\end{equation}


It follows from the definitions that as $A$-bimodules,
\begin{equation}\label{E121217_2}
A(\Om')/A(\Om)\cong \De(i)\otimes_\k \De^\op(i).
\end{equation}

Recalling (\ref{E210617}), we have a bilinear pairing $
(\cdot,\cdot)_i:\De(i)\times \De^\op(i)\to \k$ satisfying 
$$
(v_x,w_y)_i= f_i(y,x).
$$

\begin{Lemma} 
We have 
\begin{enumerate}
\item[{\rm (i)}] $(v_i,w_i)_i=1$;
\item[{\rm (ii)}] $(av,w)_i=(v,wa)_i$ for all $v\in\De(i),w\in \De^\op(i),a\in A$.
\end{enumerate}
\end{Lemma}
\begin{proof}
(i) comes from (\ref{E121217_3}).

(ii) We follow \cite[(2.3.1)]{DR}. Let $x\in X(i), y\in Y(i)$. We have 
\begin{align*}
(av_x,w_y)_i&=\sum_{x'\in X(i)}l_{x'}^x(a)(v_{x'},w_y)_i
=\sum_{x'\in X(i)}l_{x'}^x(a)f_i(y,x'),
\\
(v_x,w_ya)_i&=\sum_{y'\in Y(i)}r_{y'}^y(a)(v_x,w_{y'})_i
=\sum_{y'\in Y(i)}r_{y'}^y(a)f_i(y',x).
\end{align*}

On the other hand, by Definition~\ref{DCC}(b) and (\ref{E210617}), modulo $A(\Om)$ we have
\begin{align*}
\sum_{y'\in Y(i)}r_{y'}^y(a)f_i(y',x)e_i&\equiv \sum_{y'\in Y(i)}r_{y'}^y(a)y'x=(ya)x=y(ax)=\sum_{x'\in X(i)}l_{x'}^x(a)yx'
\\
&\equiv \sum_{x'\in X(i)}l_{x'}^x(a)f_i(y,x')e_i,
\end{align*}
so\,
$$
\sum_{y'\in Y(i)}r_{y'}^y(a)f_i(y',x)= \sum_{x'\in X(i)}l_{x'}^x(a)f_i(y,x')e_i,
$$
completing the proof. 
\end{proof}

By the lemma, 
$$
\rad \De(i):=\{v\in\De(i)\mid (v,w)_i=0\ \text{for all $w\in \De^\op(i)$}\}
$$
is a submodule of $\De(i)$. 


\begin{Lemma}\label{LDeRad}{\rm \cite[(2.4.1)]{DR}} 
Let $\k$ be a field. Then for each $i\in I$ we have that $$L(i):=\De(i)/\rad\De(i)$$ is an absolutely irreducible $A$-module. Furthermore, ignoring grading and superstructure, $\{L(i)\mid i\in I\}$ is a complete and irredundant set of irreducible $A$-modules up to an isomorphism. 
\end{Lemma}

By definition, the form $(\cdot,\cdot)_i$ is homogeneous, so 
$\rad \De(i)$ is a homogeneous submodule of $\De(i)$ and $L(i)$ is naturally a graded $A$-supermodule. We refer to the modules $L(i)$ as the {\em canonical irreducible $A$-modules}.  From Lemma~\ref{LDeRad}, we get:

\begin{Lemma}
Let $\k$ be a field. Then $$\{q^n\pi^\eps L(i)\mid i\in I,\ n\in\Z,\ \eps\in\Z/2\}$$ is a complete and irredundant set of irreducible graded $A$-supermodules up to a homogeneous isomorphism. 
\end{Lemma}

\begin{Corollary} 
Suppose that $\k$ is a local ring with the maximal ideal $\m$ and the quotient field $F=\k/\m$. Then:
\begin{enumerate}
\item[{\rm (i)}] $A/\m A\cong A\otimes_\k F$ is based quasi-hereditary $F$-superalgebra. 
\item[{\rm (ii)}] For each $i\in I$, denote the corresponding canonical irreducible $A/\m A$-module by  $L_{A/\m A}(i)$ and denote by $L_A(i)$ the $A$-module obtained from $L_{A/\m A}(i)$ by inflation. Then $$\{q^n\pi^\eps L_A(i)\mid i\in I,\ n\in\Z,\ \eps\in\Z/2\}$$ is a complete and irredundant set of irreducible graded $A$-supermodules up to a homogeneous isomorphism.
\end{enumerate} 
\end{Corollary}

If $\k$ is a local ring, we call $A$ {\em basic} if the the modules $L_{A/\m A}(i)$ are $1$-dimensional as $F$-vector spaces, equivalently if the modules $L_{A}(i)$ are free of rank $1$ as $\k$-modules. 

Let $\k$ be a field. Recalling the ring $R$ from (\ref{ER}), we can now consider {\em bigraded decomposition numbers} 
\begin{equation}\label{EDNGr}
d_{ij}(q,\pi):=\sum_{n\in\Z,\,\eps\in\Z/2}d_{ij}^{n,\eps}q^n\pi^\eps\in R \qquad(i,j\in I),
\end{equation}
where
\begin{equation}\label{EDNGr1}
d_{ij}^{n,\eps}:=[\De(i):q^n\pi^\eps L(j)]\qquad(n\in\Z,\,\eps\in\Z/2). 
\end{equation}


\begin{Lemma} 
For $i,j\in I$, we have $d_{ii}(q,\pi)=1$, and $d_{ij}(q,\pi)\neq 0$ implies $j\leq i$. 
\end{Lemma}
\begin{proof}
Denote $$\hat v_i:=v_i+\rad\De(i)\in \De(i)/\rad\De(i)=L(i).$$ Then $e_i \De(i)=\k\cdot v_i$ implies $e_i L(i)=\k\cdot \hat v_i$. Moreover, $e_j\De(i)\neq 0$ only if $j\leq i$ implies that $e_jL(i)\neq 0$ only if $j\leq i$. The result follows. 
\end{proof}

\section{Based quasi-hereditary versus split quasi-hereditary}

Throughout the section we assume that $A$ is unital. Our goal is to show that under reasonable assumptions on $\k$, the notion of based quasi-hereditary and split qusi-hereditary are the same. 

\subsection{Based quasi-hereditary algebras are split quasi-hereditary}
Assume that $\k$ is noetherian and $A$ is a graded $\k$-superalgebra, which is finitely generated projective as a $\k$-module. 
The following definition goes back to \cite{CPS,DS}, but we follow 
the version of \cite{Ro}:

\begin{Definition}\label{DSHI} 
{\rm 
A (homogeneous) ideal $J$ of $A$ is called an {\em indecomposable split heredity ideal} if the following conditions hold:
\begin{enumerate}
\item[{\rm (1)}] $A/J$ is projective as a $\k$-module;
\item[{\rm (2)}] $J$ is projective as a left $A$-module;
\item[{\rm (3)}] $J$ is idempotent, i.e. $J^2=J$;
\item[{\rm (4)}] $\End_A(J)$ is Morita equivalent to $\k$. 
\end{enumerate}
}
\end{Definition}

\begin{Definition}\label{DSQH}
{\rm 
The graded $\k$-superalgebra $A$ is {\em split quasi-hereditary} with respect to a finite partially ordered set $I$ if for every upper set $\Om\subseteq I$ there is an ideal $A(\Om)$ in $A$ such that 
\begin{enumerate}
\item[{\rm (1)}] if $\Om\subseteq \Om'$ are upper sets then $A(\Om)\subseteq A(\Om')$;
\item[{\rm (2)}]  if $\Om\subseteq \Om'$ are upper sets with $|\Om'\setminus\Om|=1$, then $A(\Om)/A(\Om)$ is an indecomposable split heredity ideal in $A/A(\Om)$. 
\end{enumerate}
}
\end{Definition}

\begin{Lemma} 
Let $\k$ be noetherian. If $A$ is based quasi-hereditary then it is split quasi-hereditary.
\end{Lemma}
\begin{proof}
By Lemma~\ref{LHI}, we have the ideals $A(\Om)$ which clearly satisfy Definition~\ref{DSQH}(1). Let $\Om\subseteq \Om'$ satisfy $\Om'\setminus\Om=\{i\}$. We need to check that the ideal  $A(\Om')/A(\Om)$ in $A/A(\Om)$ satisfies (1)--(4) of Definition~\ref{DSHI}. Note that 
$$\{xy+A(\Om')\mid (x,y)\in Z(I\setminus \Om')\}$$ is a $\k$-basis of $A/A(\Om')$, which gives (1). The property (2) follows from (\ref{E121217_2}) and (\ref{EDe}). The property (3) comes from the fact that $A(\Om')/A(\Om)$ is generated by the idempotent $e_i+A(\Om)$. Finally, by (\ref{E121217_2}) and (\ref{LEndDe}), we have $\End_A(A(\Om')/A(\Om))\cong M_m(\k)$, where $m=|Y(i)|$, which gives (4). 
\end{proof}

\subsection{Split quasi-hereditary algebras are 
based quasi-hereditary}

In this subsection, we assume that the ground ring $\k$ is noetherian and local and that $A$ is a split quasi-hereditary graded superalgebra. In particular, $A$ is a free $\k$-module of finite rank and hence Noetherian. 

In addition we assume that $A$ is {\em  semiperfect}, i.e. $A/J(A)$ is a left Artinian and homogeneous idempotents lift from $A/J(A)$ to $A$, cf. \cite[Definition 3.3]{Das}. By \cite[Theorem 3.5]{Das}, this is equivalent to $A_\0^0$ being semiperfect (in the usual sense). So, as noted in \cite[\S1]{CPS}, $A$ is semiperfect provided $\k$ is complete (local Noetherian). The proof of \cite[(1.3)]{CPS} now goes through to give:

\begin{Lemma} \label{CPS13}
Let $A$ be semiperfect. If $J_1\supseteq\dots \supseteq J_t$ are idempotent ideals in $A$ then there exist idempotents $f_1,\dots,f_t$ in  $A$ such that $J_r=Af_rA$ for all $r$ and $f_rf_s=f_sf_r=f_r$ for all $r>s$.  
\end{Lemma}

\begin{Proposition}
Assume that \(\k\) is Noetherian and local and that $A$ is a semiperfect graded $\k$-superalgebra. If \(A\) is split quasi-hereditary, then \(A\) is based quasi-hereditary.
\end{Proposition}
\begin{proof}
We may assume that $I=\{0,1,\dots,\ell\}$ for some $\ell\in\Z_{>0}$ and $0<1<\dots<\ell$ is a total order refining the given partial order on $I$. Then $\Om_i:=\{i,i+1,\dots,\ell\}$ is an upper set for any $i\in I$, and we have a chain 
$$
I=\Omega_0\supseteq \Om_1\supseteq\dots\supseteq \Om_\ell\supseteq \Omega_{\ell+1}:=\varnothing
$$ 
with \(\Omega_i \backslash \Omega_{i+1} = \{i\}\) for \(i \in I\). By Lemma \ref{CPS13} there exist idempotents \(f_0, \ldots, f_\ell\) such that \(A(\Omega_i)=Af_iA \), and \(f_if_j = f_jf_i = f_i\) whenever \(i>j\). Define \(e_\ell:=f_\ell\) and \(e_i := f_i - f_{i+1}\) for $i=0,1,\dots,\ell-1$. Then for all \(i,j \in I\), we have \(e_ie_j = \delta_{ij}e_i\), and \(f_i= e_i + \cdots + e_\ell\).

Let \(i \in I\), \(\tilde{A}:=A/A(\Omega_{i+1})\) and $\tilde a:=a+A(\Omega_{i+1})\in \tilde A$ for $a\in A$. 
It follows from Definition~\ref{DSHI}(1) that $\tilde A$ is projective as a $\k$-module. Moreover, 
$A(\Omega_i)/A(\Omega_{i+1})$ is projective as an \(\tilde{A}\)-module. Since 
$$A(\Omega_i)/A(\Omega_{i+1}) = \tilde{A}\tilde{f}_i\tilde{A} = \tilde{A}\tilde{e}_i\tilde{A}$$ 
by the previous paragraph,  \cite[Statement 7]{DlRi} implies that the multiplication map
\begin{align*}
m: \tilde{A}\tilde{e}_i \otimes_{\tilde{e}_i\tilde{A}\tilde{e}_i} \tilde{e}_i\tilde{A} \to \tilde{A} \tilde{e}_i \tilde{A}
\end{align*}
is an isomorphism of $\tilde A$-bimodules. By \cite[Lemma 4.5, Proposition 4.7]{Ro}, we have that 
\begin{align*}
\tilde{e}_i\tilde{A}\tilde{e}_i \cong \End_{\tilde{A}}(\tilde{A}\tilde{e}_i)^\op = \End_{A}(\tilde{A}\tilde{e}_i)^\textup{op} \cong \k,
\end{align*}
so \(\tilde{e}_i \tilde{A} \tilde{e}_i = \k \tilde{e}_i\), and \(A(\Omega_i)/A(\Omega_{i+1}) \cong \tilde{A}\tilde{e}_i \otimes_\k \tilde{e}_i \tilde{A}\). 


The left \(\tilde A\)-module \(\tilde{A}\tilde{e}_i\) is projective as an \(\tilde A\)-module, hence projective as a \(\k\)-module. Writing \(e_* := 1-e_0 - \cdots -e_\ell\), we have
\begin{align*}
\tilde{A}\tilde{e}_i = \tilde{e}_0\tilde{A}\tilde{e}_i \oplus \cdots \oplus \tilde{e}_\ell \tilde{A}\tilde{e}_i \oplus \tilde{e}_*\tilde{A}\tilde{e}_i.
\end{align*}
Each of the summands above is projective as a \(\k\)-module, hence is free as a \(\k\)-module since \(\k\) is local. Then there exists a set of elements \(X(i) \subset A_{\textup{hom}}\) such that:
\begin{itemize}
\item \(e_i \in X(i)\);
\item \(\{\tilde{x} \mid x \in X(i)\}\) is a \(\k\)-basis for \(\tilde{A}\tilde{e}_i\);
\item For all \(x \in X(i)\), we have \(x=e_t x e_i\) for some \(t \in \{0,\ldots, \ell, *\}\).
\end{itemize}
In similar fashion we may choose a set of elements \(Y(i) \subset A_{\textup{hom}}\) such that:
\begin{itemize}
\item \(e_i \in Y(i)\);
\item \(\{\tilde{y} \mid y \in Y(i)\}\) is a \(\k\)-basis for \(\tilde{e}_i\tilde{A}\);
\item For all \(y \in Y(i)\), we have \(y=e_i y e_t\) for some \(t \in \{0,\ldots, \ell, *\}\).
\end{itemize}
Since $m$ is an isomorphism, \(\{\tilde{x}\tilde{y} \mid x \in X(i),\ y \in Y(i)\}\) is a \(\k\)-basis for \(\tilde{A}\tilde{e}_i \tilde{A} = A(\Omega_i)/A(\Omega_{i+1})\), for all \(i \in I\), which implies that \(\{xy \mid i \in I,\ x \in X(i),\ y \in Y(i)\}\) is a basis for \(A\). 
The remaining conditions of Definition~\ref{DCC} are now easily checked. For example, $e_ix=\de_{x,e_i}x$ for $x\in X(i)$ follows from $\tilde{e}_i\tilde{A}\tilde{e}_i\cong \k \tilde e_i$. 
Thus \(\{I,\ \bigsqcup_i X(i),\ \bigsqcup_i Y(i)\}\) constitutes based quasi-hereditary data for \(A\). 
\end{proof}

\section{Further properties}\label{SQHAReg}
Let $A$ be a based quasi-hereditary $\k$-superalgebra with heredity data $I,X,Y$.  

\subsection{Involution and idempotent truncation}\label{SSIdTr}
If $e\in A$ is a homogeneous idempotent, we consider the idempotent truncation $\bar A:=eAe$, and denote $\bar a:= eae\in\bar A$ for $a\in A$. We say that $e$ is {\em adapted} (with respect to the given heredity data) if for all $i\in I$ there exist subsets $\bar X(i)\subseteq X(i)$ and $\bar Y(i)\subseteq Y(i)$ such that for all $(x,y)\in Z(i)$ we have:
\begin{equation}\label{E230517}
ex= 
\left\{
\begin{array}{ll}
x &\hbox{if $x\in\bar X(i)$,}\\
0 &\hbox{otherwise,}
\end{array}
\right.
\qquad\text{and}\qquad\ \ 
ye=
\left\{
\begin{array}{ll}
y &\hbox{if $y\in\bar Y(i)$,}\\
0 &\hbox{otherwise.}
\end{array}
\right.
\end{equation}
Setting 
\begin{equation}\label{E200717}
\bar I:=\{i\in I\mid \bar X(i)\neq \emptyset\neq \bar Y(i)\},
\end{equation} 
the {\em $e$-truncation of $B$} is defined to be 
\begin{equation}\label{ETrunc}
\bar B:=\{xy\mid i\in \bar I, x\in\bar X(i), y\in\bar Y(i)\}. 
\end{equation} 
We say that $e$ is {\em strongly adapted} if it is adapted and $
ee_i=e_ie= e_i
$
for all $i\in\bar I$.

\begin{Lemma} \label{L200517_4} 
Let $e\in A$ be an adapted idempotent. 
\begin{enumerate}
\item[{\rm (i)}] The $e$-truncation $\bar B$ is a standard basis of $\bar A$ in the sense of Definition~\ref{DSB}.  

\item[{\rm (ii)}] If $\tau$ is a standard anti-involution of $A$ such that $\tau(e)=e$, then $\bar B$ is a cellular basis of $\bar A$ with respect to the restriction $\tau|_{\bar A}$. 

\item[{\rm (iii)}] If $e$ is strongly adapted then $\bar A$ is based quasi-hereditary with heredity data $\bar I$, $\bar X:=\bigsqcup_{i\in \bar I}\bar X(i)$, $\bar Y:=\bigsqcup_{i\in \bar I}\bar Y(i)$. 
\end{enumerate}
\end{Lemma}
\begin{proof}
(i) follows from $xy=exye$. 
To check (ii) one needs to observe that $ex=x$ if and only if $y(x)e=y(x)$, and so $\bar Y(i)=\{y(x)\mid x\in\bar X(i)\}$. Part (iii) is clear. 
\end{proof}

\begin{Remark} 
{\rm 
Let $e\in A$ be an adapted idempotent. For $i\in I$,  consider the $\bar A$-module 
$
\bar\De(i):=e\De(i). 
$ 
If $\tau$ is a standard anti-involution of $A$ with $\tau(e)=e$ then by Lemma~\ref{L200517_4}(ii), $\bar A$ is cellular and $\{\bar\De(i)\mid i\in \bar I\}$ are the cell modules for $\bar A$. If $e$ is strongly adapted then by Lemma~\ref{L200517_4}(iii), $\bar A$ is quasi-hereditary and $\{\bar\De(i)\mid i\in \bar I\}$ are the standard modules for $\bar A$. 
}
\end{Remark}

\begin{Remark} 
{\rm 
Given a cellular algebra $\bar A$ with cellular basis $\bar B$ and a subalgebra $\bar\a\subseteq \bar A_\0$, is there a based quasi-hereditary algebra $A$ with heredity basis $B$, a standard anti-involution $\tau$ and $\tau$-invariant adapted idempotent $e$ such that $\bar A=eAe$, $\bar\a=e\a e$, and $\bar B$ is the $e$-truncation of $B$? We do not know if this converse of Lemma~\ref{L200517_4}(ii) always holds true. This question seems to be related to problems studied in \cite{Ro,DlRi2,Ko,Aus}. 
}
\end{Remark}

\begin{Lemma} 
Let \(\k\) be a field, and $e\in A$ be an adapted idempotent. 
\begin{enumerate}
\item[{\rm (i)}] $eL(i)=0$ if and only if $e\De(i)\subseteq \rad\De(i)$.
\item[{\rm (ii)}] $eL(i)=0$ if and only if $yex\in A^{>i}$ for all $x\in X(i)$ and $y\in Y(i)$.

\item[{\rm (iii)}] $eL(i)=0$ if and only if $yx\in A^{>i}$ for all $x\in \bar X(i)$ and $ y\in\bar Y(i)$.

\item[{\rm (iv)}] $eL(i)=0$ for all $i\in I\setminus \bar I$. 
\end{enumerate}
\end{Lemma}
\begin{proof}
Part (i) is clear. By part (i), $eL(i)=0$ if and only if $ev_x\in\rad \De(i)$ for all $x\in X(i)$. Recalling the definition of the form $(\cdot,\cdot)_i$, this is equivalent to $yex\in A^{>i}$, proving part (ii). Pari (iii) follows from part (ii) since $ex=\de_{\{x\in \bar X\}}x$ and $ye=\de_{\{y\in \bar Y\}}y$. Finally, if $i\in I\setminus \bar I$ then $\bar X(i)=\varnothing$ or $\bar Y(i)=\varnothing$ (or both). So part (iv) follows from part (iii). 
\end{proof}

\begin{Corollary} 
Let \(\k\) be a field, and $e\in A$ be an adapted idempotent. Then there exists a subset $\bar I'\subseteq\bar I$ such that $\{eL(i)\mid i\in \bar I'$ is a complete and irredundant set of irreducible $\bar A$-modules up to isomorphism. 
\end{Corollary}

\subsection{Conformity}

We now turn to more subtle additional properties of heredity data, which have to do with the super-structure. 
Recalling (\ref{EH}), we have sets $B_\eps, X(i)_\eps, Y_\eps$ etc. 


\begin{Definition}\label{D290517}
{\rm 
Suppose that $\fa\subseteq A_\0$ is a subalgebra. The heredity data $I,X,Y$ of $A$ is {\em $\fa$-conforming} if $I,X_\0,Y_\0$ is a heredity data for $\fa$. 
}
\end{Definition}

If the heredity data $I,X,Y$ of $A$ is $\fa$-conforming then $\fa$ is recovered as follows:
$$ 
\fa=\spa(xy\mid i\in I,\ x\in X(i)_\0,\ y\in Y(i)_\0).
$$
So sometimes we will just speak of a {\em conforming heredity data}. Even though in some sense $\fa$ is redundant in the definition of conormity, it is often convenient to use it. For example, in \cite{greenTwo}, we will construct generalized Schur algebras $T^A_\fa(n,d)$, which will only depend on $A$ and $\fa$, but not on $I,X,Y$.

Recall that we have standard $A$-modules $\De(i)$ and simple $A$-modules $L(i)$ (if $\k$ is a field). If the heredity data $I,X,Y$ of $A$ is $\fa$-conforming then by definition $\fa$ is also based quasi-hereditary and has its own standard $\fa$-modules $\De_\fa(i)$ and simple $\fa$-modules $L_\fa(i)$ (if $\k$ is a field). 


We describe an additional property which implies  conformity. This property is readily checked in some important examples and will be preserved under formation of  the generalized Schur algebra $T^A_\fa(n,d)$. The following is easy to see:

\begin{Lemma} 
Suppose that $A$ possesses a $(\Z/2\times \Z/2)$-grading 
$
A=\bigoplus_{\eps,\de\in\Z/2} A_{\eps,\de}
$
such that the following conditions hold:
\begin{enumerate}
\item[{\rm (1)}] $A_{\eps,\de}A_{\eps',\de'}\subseteq A_{\eps+\eps',\de+\de'}$ for all $\eps,\de,\eps',\de'\in\Z/2$;
\item[{\rm (2)}] For all $\eps\in\Z/2$, we have $A_\eps=\bigoplus_{\eps'+\eps''=\eps} A_{\eps',\eps''}$. 
\item[{\rm (3)}] $X_\eps\subseteq A_{\eps,\0}$ and $Y_\eps\subseteq A_{\0,\eps}$ for all $\eps\in\Z/2$. 
\end{enumerate}
Then the heredity data $I,X,Y$ is $\fa$-conforming for $\fa=A_{\0,\0}$. 
\end{Lemma}

\subsection{Morita equivalence}
Throughout the section, we assume that $\k$ is 
local. 
We also assume that \(A\) is a unital based quasi-hereditary graded \(\k\)-superalgebra with heredity data \(I,X,Y\) which is $\fa$-conforming for a unital subalgebra $\a$, in particular, $I,X_\0,Y_\0$ is a heredity data for $\a$ and $1_\a=1_A$.

Our goal is to find an idempotent $f\in\fa$ such that  $\bar A:=fAf$ is based quasi-hereditary with $\bar\fa$-conforming hereditary data, where $\bar \fa:=f\fa f$ is basic and the functors
$$
\funF_A:\mod{A}\to\mod{\bar A},\ V\mapsto f V
\quad\text{and}\quad 
\funF_\mathfrak{a}:\mod{\mathfrak{a}}\to\mod{\bar{\mathfrak{a}}},\ V\mapsto fV
$$ 
are equivalences of categories, such that 
\begin{align*}
&\funF_A(L_A(i))\cong L_{\bar{A}}(i), 
&\funF_A(\De_A(i))\cong \De_{\bar{A}}(i),
\\ 
&\funF_\mathfrak{a}(L_\mathfrak{a}(i))\cong L_{\bar{\mathfrak{a}}}(i),  
&\funF_\mathfrak{a}(\De_\mathfrak{a}(i))\cong \De_{\bar{\mathfrak{a}}}(i).
\end{align*}

The first step allows us to reduce to the situation where $\sum_{i\in I}e_i=1_A=1_\fa$: 

\begin{Lemma} \label{L221217} 
Let $e:=\sum_{i\in I} e_i$. Then $\bar A:=eAe$ is based quasi-hereditary with $\bar\fa$-conforming hereditary data, where $\bar \fa:=e\fa e$ and the functors
$$
\funF_A:\mod{A}\to\mod{\bar A},\ V\mapsto e V
\quad\text{and}\quad 
\funF_\mathfrak{a}:\mod{\mathfrak{a}}\to\mod{\bar{\mathfrak{a}}},\ V\mapsto eV
$$ 
are equivalences of categories, such that 
\begin{align*}
&\funF_A(L_A(i))\cong L_{\bar{A}}(i), 
&\funF_A(\De_A(i))\cong \De_{\bar{A}}(i),
\\ 
&\funF_\mathfrak{a}(L_\mathfrak{a}(i))\cong L_{\bar{\mathfrak{a}}}(i),  
&\funF_\mathfrak{a}(\De_\mathfrak{a}(i))\cong \De_{\bar{\mathfrak{a}}}(i).
\end{align*}
\end{Lemma}
\begin{proof}
This follows using Lemma~\ref{L200517_4} since $e$ is strongly adapted. 
\end{proof}

\begin{Lemma}\label{makeprim}
There exists an \(\mathfrak{a}\)-conforming heredity data \(I,X',Y'\) for \(A\) with  the same ideals \(A(\Omega)\) and \(\mathfrak{a}(\Omega)\), and such that the new initial elements \(\{e'_i \mid i \in I\}\) are primitive idempotents in \(\mathfrak{a}\)  satisfying $e_ie_i'=e_i'=e_i'e_i$ 
and $e_i'\equiv e_i\pmod{\a^{>i}}$ for all $i\in I$.
\end{Lemma}

\begin{proof}
Let $i\in I$. Set $\tilde \fa:=\fa/\fa^{>i}$ and $\tilde a:=a+\fa^{>i}\in\tilde \fa$ for $a\in \fa$. 
Then $\tilde e_i$ is a primitive idempotent in $\tilde \fa$ since 
$\End_{\tilde \fa}(\tilde \fa \tilde e_i)\cong \tilde e_i \tilde \fa \tilde e_i\cong \k$ is local. So if $e_i=e_i^1+\dots+e_i^r$ is a sum of orthogonal primitive idempotents in $\fa$ then there is exactly one $t$ with $1\leq t\leq r$ and $\tilde e_i= \tilde e_i^t$. We set $e_i':=e_i^t$. Note that  $e_ie_i'=e_i'=e_i'e_i$, hence $e_i'e_j'=0$ for $i\neq j$. 

Let $\Om$ be an upper set of $I$. 
It easily follows that $A(\Om)$, which by Lemma~\ref{LHI} is the ideal of $A$ generated by $\sum_{i\in\Om}e_i$, is also generated by $\sum_{i\in\Om}e_i'$. Similarly, $\fa(\Om)$ is the ideal of $\fa$ generated by $\sum_{i\in\Om}e_i'$.

We have that $\fa^{\geq i}/\fa^{>i}$ is projective as an \(\tilde \fa\)-module,  
$
\fa^{\geq i}/\fa^{>i} = \tilde \fa\tilde e_i\tilde \fa=\tilde \fa\tilde e_i'\tilde \fa
$  and $\tilde e_i' \tilde \fa \tilde e_i'=\tilde e_i \tilde \fa \tilde e_i\cong \k$. So 
\cite[Statement 7]{DlRi} implies that the multiplication map
$$
m: \tilde{\fa}\tilde{e}_i' \otimes_{\k} \tilde{e}_i'\tilde{\fa} \to \tilde{\fa} \tilde{e}_i' \tilde{\fa}
$$
is an isomorphism of $\tilde \fa$-bimodules. By definition, $\tilde \a \tilde e_i'=\tilde \a \tilde e_i$ has $\k$-basis $\{\tilde x\mid x\in X(i)_\0\}$, $\tilde A_\1 \tilde e_i'=(\tilde A \tilde e_i')_\1=(\tilde A \tilde e_i)_\1$ has $\k$-basis $\{\tilde x\mid x\in X(i)_\1\}$, and $\tilde A \tilde e_i'=\tilde A \tilde e_i=\tilde \a \tilde e_i'\oplus \tilde A_\1 \tilde e_i'$ as $\k$-modules. Let 
$$e_*' := 1_A-\sum_{i\in I}e_i'.$$ 
Since $1_A=1_\fa$, we have $e_*'\in\a$. Note that  
\begin{align*}
\tilde{\a}\tilde{e}_i' = \bigoplus_{j\in I\sqcup\{*\}}\tilde{e}_j'\tilde{\a}\tilde{e}_i' 
\quad\text{and}\quad
\tilde{A}_\1\tilde{e}_i' = \bigoplus_{j\in I\sqcup\{*\}}\tilde{e}_j'\tilde{A}_\1\tilde{e}_i'.
\end{align*}
Each of the summands above is projective, hence free, as a \(\k\)-module. So there exists a set of elements \(X'(i)=X'(i)_\0\sqcup X'(i)_\1\) such that:
\begin{itemize}
\item \(e_i' \in X(i)_\0\);
\item \(\{\tilde{x} \mid x \in X'(i)_\0\}\) is a \(\k\)-basis for \(\tilde{\a}\tilde{e}_i'\) and \(\{\tilde{x} \mid x \in X'(i)_\1\}\) is a \(\k\)-basis for \(\tilde{A}_\1\tilde{e}_i'\);
\item For all \(x \in X'(i)\), we have \(x=e_j' x e_i'\) for some \(j \in 
I\sqcup\{*\}
\).
\end{itemize}
In similar fashion we may choose a set of elements \(Y'(i)=Y'(i)_\0\sqcup Y'(i)_\1\) such that:
\begin{itemize}
\item \(e_i' \in Y'(i)_\0\);
\item \(\{\tilde{y} \mid y \in Y'(i)_\0\}\) is a \(\k\)-basis for \(\tilde{e}_i'\tilde{\a}\) and \(\{\tilde{y} \mid y \in Y'(i)_\1\}\) is a \(\k\)-basis for \(\tilde{e}_i'\tilde{A}_\1\);
\item For all \(y \in Y'(i)\), we have \(y=e_i' y e_j'\) for some \(j \in 
I\sqcup\{*\}
\).
\end{itemize}

Since $m$ is an isomorphism, \(\{\tilde{x}\tilde{y} \mid x \in X'(i),\ y \in Y'(i)\}\) is a \(\k\)-basis for \(\tilde{A}\tilde{e}_i' \tilde{A} = A^{\geq i}/A^{>i}\) and 
\(\{\tilde{x}\tilde{y} \mid x \in X'(i)_\0,\ y \in Y'(i)_\0\}\) is a \(\k\)-basis for \(\tilde{\a}\tilde{e}_i' \tilde{\a} = \fa^{\geq i}/\fa^{>i}\). Doing this for all $i\in I$, we deduce that 
\(\{xy \mid i \in I,\ x \in X'(i), y \in Y'(i)\}\) is a basis for \(A\) and \(\{xy \mid i \in I,\ x \in X'(i)_\0, y \in Y'(i)_\0\}\) is a basis for \(\a\). 
The remaining conditions of Definitions~\ref{DCC} and \ref{D290517} are now easily checked. Thus \(\{I,\ \bigsqcup_i X'(i),\ \bigsqcup_i Y'(i)\}\) is an $\a$-conforming heredity data for \(A\). 
\end{proof}

In Lemma~\ref{makeprim}, we have obtained the condition that all the heredity ideals \(A(\Omega)\) are the same for the two heredity bases coming from $(I,X,Y)$ and $(I,X',Y')$. This implies that the standard modules $\De_A(i)$ and hence the simple modules $L_A(i)$ are unchanged when we pass from $(I,X,Y)$ and $(I,X',Y')$. The similar statement holds for $\De_\fa(i)$ and $L_\fa(i)$. 

For a strongly adapted idempotent $e\in A$, recall the notation $\bar X(i),\bar Y(i)$ from (\ref{E230517}),\,(\ref{E200717}). These will be applied for the idempotent $f$ appearing in the following theorem:

\begin{Theorem}\label{MorBasic}
Let $\k$ be local and \(A\) be a unital based quasi-hereditary graded \(\k\)-superalgebra with $\fa$-conforming heredity data $(I,X,Y)$ for a unital subalgebra $\a$. Then there exists an $\fa$-conforming heredity data $(I,X',Y')$ with the same ideals $A(\Om)$ and $\fa(\Om)$ and such that the new initial elements $\{e'_i \mid i \in I\}$ are primitive idempotents in \(\mathfrak{a}\)  satisfying $e_ie_i'=e_i'=e_i'e_i$ 
and $e_i'\equiv e_i\pmod{\a^{>i}}$ for all $i\in I$. Moreover, setting $f:=\sum_{i\in I}e_i'$, we have:
\begin{enumerate}
\item[{\rm (i)}] $f$ is strongly adapted with respect to $(I,X',Y')$, so that $\bar A$ is based quasi-hereditary with heredity data $(I,\bar X',\bar Y')$. 
\item[{\rm (ii)}] $(I,\bar X',\bar Y')$ is $\bar\fa$-conforming;
\item[{\rm (iii)}] $\bar \a$ is basic and if \(A_{\overline{1}} \subset J(A)\) then \(\bar{A}\) is a basic as well;
\item[{\rm (iv)}] The functors
$$
\funF_A:\mod{A}\to\mod{\bar A},\ V\mapsto f V
\quad\text{and}\quad 
\funF_\mathfrak{a}:\mod{\mathfrak{a}}\to\mod{\bar{\mathfrak{a}}},\ V\mapsto fV
$$ 
are equivalences of categories, such that 
\begin{align*}
&\funF_A(L_A(i))\cong L_{\bar{A}}(i), &\funF_A(\De_A(i))\cong \De_{\bar{A}}(i),\\
 &\funF_\mathfrak{a}(L_\mathfrak{a}(i))\cong L_{\bar{\mathfrak{a}}}(i), 
 &\funF_\mathfrak{a}(\De_\mathfrak{a}(i))\cong \De_{\bar{\mathfrak{a}}}(i).
\end{align*}
\end{enumerate}
\end{Theorem}
\begin{proof}
Let $e=\sum_{i\in I}e_i$. By Lemma~\ref{L221217}, the algebra $eAe$ satisfies the assumptions of Lemma~\ref{makeprim}. The application of that lemma yields a conforming heredity data $(I,X'',Y'')$ in $eAe$ with initial elements $\{e_i''\mid i\in I\}$. To extend it to the needed heredity data $(I,X',Y')$ for $A$ define
\begin{align*}
X'&:=X''\sqcup \bigsqcup_{i\in I}\{xe_i''\mid x\in X(i)\ \text{with}\ ex=0\},
\\
Y'&:=Y''\sqcup \bigsqcup_{i\in I}\{e_i''y\mid y\in Y(i)\ \text{with}\ ye=0\}.
\end{align*}
It is easy to see that this new heredity data with initial elements $e_i'=e_i''$ 
satisfies the required conditions. 
\end{proof}

\subsection{Examples}\label{SSZig}

Our two main examples of based quasi-hereditary algebras are the classical {\em Schur algebra} $S(n,d)$  and the {\em extended zigzag algebra} $\EZig$. 

The classical Schur algebra with trivial grading and superalgebra structures has the basis $\{Y^\la_{S,T}\}$  of codeterminants constructed in \cite{GreenCod}. It is essentially checked in \cite{GreenCod} that $S(n,d)$ with the codeterminant basis is a  based quasi-hereditary algebra with perfect heredity data and standard anti-involution. So is the extended zigzag algebra, which we define next. 

Given \(n \geq d\), let \(\la = (1^d)\), and  let \(T^\la\) be the \(\la\)-tableau with the entry \(r\) in the \(r\)th row. Define
\begin{align*}
e:= Y^{\la}_{T^\la, T^\la} = \xi_{1 \cdots d, 1 \cdots d}.
\end{align*}
Then \(e\) is an adapted idempotent, and \(eS(n,d)e \cong \k \mathfrak{S}_d\). Thus
\begin{align*}
\{ eY^{\la}_{S,T}e \mid eY^{\la}_{S,T}e \neq 0\}
 \end{align*}
 defines a cellular basis for \(\mathfrak{S}_d\), known as a {\em Murphy basis}.

Fix $\ell\geq 1$ and set 
$$
I:=\{0,1,\dots,\ell\},\quad J:=\{0,\dots,\ell-1\}.
$$

Let $\Gamma$ be the quiver with vertex set $I$ and arrows $\{a_{j,j+1},a_{j+1,j}\mid j\in J\}$ as in the picture: 
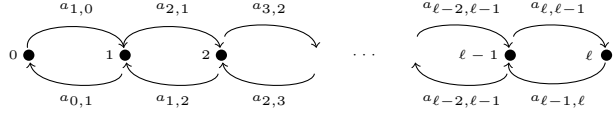
\begin{align*}
\begin{braid}\tikzset{baseline=3mm}
\coordinate (0) at (-4,0);
\coordinate (1) at (0,0);
\coordinate (2) at (4,0);
\coordinate (3) at (8,0);
\coordinate (6) at (12,0);
\coordinate (L1) at (16,0);
\coordinate (L) at (20,0);
\draw [thin, black,->,shorten <= 0.1cm, shorten >= 0.1cm]   (0) to[distance=1.5cm,out=100, in=100] (1);
\draw [thin,black,->,shorten <= 0.25cm, shorten >= 0.1cm]   (1) to[distance=1.5cm,out=-100, in=-80] (0);
\draw [thin, black,->,shorten <= 0.1cm, shorten >= 0.1cm]   (1) to[distance=1.5cm,out=100, in=100] (2);
\draw [thin,black,->,shorten <= 0.25cm, shorten >= 0.1cm]   (2) to[distance=1.5cm,out=-100, in=-80] (1);
\draw [thin,black,->,shorten <= 0.25cm, shorten >= 0.1cm]   (2) to[distance=1.5cm,out=80, in=100] (3);
\draw [thin,black,->,shorten <= 0.25cm, shorten >= 0.1cm]   (3) to[distance=1.5cm,out=-100, in=-80] (2);
\draw [thin,black,->,shorten <= 0.25cm, shorten >= 0.1cm]   (6) to[distance=1.5cm,out=80, in=100] (L1);
\draw [thin,black,->,shorten <= 0.25cm, shorten >= 0.1cm]   (L1) to[distance=1.5cm,out=-100, in=-80] (6);
\draw [thin,black,->,shorten <= 0.25cm, shorten >= 0.1cm]   (L1) to[distance=1.5cm,out=80, in=100] (L);
\draw [thin,black,->,shorten <= 0.1cm, shorten >= 0.1cm]   (L) to[distance=1.5cm,out=-100, in=-100] (L1);
\blackdot(-4,0);
\blackdot(0,0);
\blackdot(4,0);
\blackdot(16,0);
\blackdot(20,0);
\draw(-4,0) node[left]{$0$};
\draw(0,0) node[left]{$1$};
\draw(4,0) node[left]{$2$};
\draw(10,0) node {$\cdots$};
\draw(13.4,0) node[right]{$\ell-1$};
\draw(18.65,0) node[right]{$\ell$};
\draw(-2,1.2) node[above]{$ a_{1,0}$};
\draw(2,1.2) node[above]{$ a_{2,1}$};
\draw(6,1.2) node[above]{$ a_{3,2}$};
\draw(14,1.2) node[above]{$ a_{\ell-2,\ell-1}$};
\draw(18,1.2) node[above]{$ a_{\ell,\ell-1}$};
\draw(-2,-1.2) node[below]{$ a_{0,1}$};
\draw(2,-1.2) node[below]{$ a_{1,2}$};
\draw(6,-1.2) node[below]{$ a_{2,3}$};
\draw(14,-1.2) node[below]{$ a_{\ell-2,\ell-1}$};
\draw(18,-1.2) node[below]{$ a_{\ell-1,\ell}$};
\end{braid}
\end{align*}

The {\em extended zigzag algebra $\EZig$} is the path algebra $\k\Gamma$ modulo the following relations:
\begin{enumerate}
\item All paths of length three or greater are zero.
\item All paths of length two that are not cycles are zero.
\item All length-two cycles based at the same vertex are equivalent.
\item $ a_{\ell,\ell-1} a_{\ell-1,\ell}=0$.
\end{enumerate}
Length zero paths yield the standard idempotents $\{ e_i\mid i\in I\}$ with $ e_i  a_{i,j} e_j= a_{i,j}$ for all admissible $i,j$. The algebra $\EZig$ is graded by the path length: 
$\EZig=\EZig^0\oplus \EZig^1\oplus \EZig^2.
$ 
We also consider $\EZig$ as a superalgebra with 
$\EZig_\0=\EZig^0\oplus \EZig^2\quad \text{and}\quad \EZig_\1=\EZig^1.
$ 

Define
$$
 c_j:= a_{j,j+1} a_{j+1,j} \qquad (j\in J).
$$
The algebra $\EZig$ has an anti-involution $\tau$ with
$$
\tau( e_i)=  e_i,\quad  \tau(a_{ij})=  a_{ji},\quad  \tau(c_j) = c_j.
$$

We consider the total order on $I$
 given by $0<1<\dots<\ell$. For $i\in I$, we set 
 $$X(i):=
 \left\{
\begin{array}{ll}
\{e_i,a_{i-1,i}\}   &\hbox{if $i>0$,}\\
\{e_0\}  &\hbox{if $i=0$,}
 \end{array}
 \right.
 \quad
 Y(i):=
 \left\{
\begin{array}{ll}
\{e_i,a_{i,i-1}\}   &\hbox{if $i>0$,}\\
\{e_0\}  &\hbox{if $i=0$.}
 \end{array}
 \right.
 $$
With respect to this data we have:

\begin{Lemma} \label{LAQH} 
The graded superalgebra $\EZig$ is a basic based quasi-hereditary with perfect heredity data and standard anti-involution $\tau$.\end{Lemma}
\begin{proof}
This is well-known and easy to check.
\end{proof}

Note that 
$$B_\1=\{ a_{j,j+1}, a_{j+1,j}\mid j\in J\},\ \ B_{\0}=\{ e_i\mid i\in I\}\sqcup \{ c_j\mid j\in J\}.
$$


Let \(e:= e_0 + \cdots + e_{\ell-1} \in Z\). Note that \(e\) is an adapted idempotent, and \(\tau(e) = e\) so the {\em zigzag algebra} \(\overline{Z}:=e Z e \subset Z\) is a cellular algebra with involution \(\tau|_{\overline{Z}}\), and cellular basis
\begin{align*}
\overline{B} = \{xy\mid i \in I, x \in \overline{X}(i), y \in \overline{Y}(i)\},
\end{align*}
where \(\overline{X}(\ell) = \{a_{\ell-1,\ell}\}\), \(\overline{Y}(\ell)=\{a_{\ell,\ell-1}\}\), and \(\overline{X}(i) = X(i)\), \(\overline{Y}(i) = Y(i)\) for all \(i \in J\).

Note that, when \(\k\) is a field, we have \(eL(\ell) = 0\), and \(eL(j) = L(j)\) for all \(j \in J\), so the standard \(\overline{\EZig}\)-modules are \(\{\bar{\Delta}(i) = e\Delta(i) \mid i \in I\}\), and the simple \(\overline{\EZig}\)-modules are \(\{\bar{L}(j) = eL(j) \mid j \in J\}\). The following lemma is easily checked.

\begin{Lemma}\label{zigdecomp} Let \(\k\) be a field. Let \(i \in I\), and \(j \in I\) (resp. \(j \in J\)). Then the graded decomposition numbers for standard \(\EZig\)-modules (resp. \(\overline{\EZig}\)-modules) are given by
\begin{align*}
d_{i,j} = 
\delta_{i,j} +\delta_{i-1,j} q \pi.
\end{align*}
\end{Lemma}

For integers \(n,m\), consider the matrix superalgebra \(M_{n|m}(\k)\) of rank \(n|m\), with entries in \(\k\). For \(r,s \in [1,n+m]\), let \(E_{r,s}\) be the matrix with \(1\) in the \((r,s)\)-th component, and zeros elsewhere. We have
\begin{align*}
\overline{E}_{r,s} := 
\begin{cases}
\bar 0 & \textup{if }r,s \leq n \textup{ or } r,s >n,\\
\bar 1 & \textup{otherwise}.
\end{cases}
\end{align*}
and
\begin{align*}
\textup{deg}(E_{r,s}) = r-s.
\end{align*}
Then \(B:=\{E_{r,s} \mid r,s \in [1,n+m]\}\) constitutes a homogeneous basis for \(M_{n|m}(\k)\).

Now, let \(I = \{\bullet\}\) be the singleton set, and define:
\begin{align*}
e_\bullet:=E_{1,1}, \hspace{5mm} X(\bullet):=\{E_{r,1} \mid r \in [1,n+m]\}, \hspace{5mm} Y(\bullet):=\{E_{1,s} \mid s \in [1, n+m]\}.
\end{align*}
Then \((I,X,Y)\) constitutes conforming heredity data for \(M_{n|m}(\k)\) with heredity basis \(B\).

\end{document}